\documentclass[a4paper]{article}

\usepackage[latin1]{inputenc}
\usepackage[T1]{fontenc}
\usepackage{ae,aecompl}
\usepackage{amsmath,amsfonts}

\usepackage[tableposition=top,font=small,labelfont=bf,format=hang]{caption}
\usepackage{subfig}

\usepackage{booktabs}
\usepackage[sort&compress,numbers]{natbib}
\usepackage[dvips,colorlinks=false]{hyperref}
\hypersetup{
pdfauthor = {Lars Eirik Danielsen and Matthew G. Parker},
pdftitle = {Interlace Polynomials: Enumeration, Unimodality, and Connections to Codes},
pdfkeywords = {Interlace polynomials, Local complementation, Unimodality, Circle graphs, Quantum codes, Quantum graph states}
}
\usepackage[dvips]{graphicx}

\usepackage{amsthm}
\newtheorem{thm}{Theorem}
\newtheorem{prop}[thm]{Proposition}

\theoremstyle{definition}
\newtheorem{defn}[thm]{Definition}
\newtheorem{exmp}[thm]{Example}
\theoremstyle{remark}

\newtheorem{con}{Conjecture}

\title{Interlace Polynomials: Enumeration, Unimodality, and Connections to Codes}
\author{Lars Eirik Danielsen\thanks{Department of Informatics, University of Bergen, 
PO Box 7803, \mbox{N-5020} Bergen, Norway.\hfill
\texttt{\{\href{mailto:larsed@ii.uib.no}{larsed},\href{mailto:matthew@ii.uib.no}{matthew}\}@ii.uib.no}$\quad$
\texttt{http://www.ii.uib.no/\~{}\{\href{http://www.ii.uib.no/\~larsed}{larsed},\href{http://www.ii.uib.no/\~matthew}{matthew}\}}} \and Matthew G. Parker\footnotemark[1]}
\date{October 1, 2009}

\begin{document}

\maketitle

\begin{abstract}
The \emph{interlace polynomial} $q$ was introduced by
Arratia, Bollobás, and Sorkin. It encodes many properties of the orbit of 
a graph under \emph{edge local complementation} (ELC). The interlace polynomial $Q$,
introduced by Aigner and van der Holst, similarly contains information about the
orbit of a graph under \emph{local complementation} (LC).
We have previously classified LC and ELC orbits, and now give an enumeration of the corresponding
interlace polynomials of all graphs of order up to 12.
An enumeration of all \emph{circle graphs} of order up to 12 is also given.
We show that there exist graphs of all orders greater than 9 with
interlace polynomials $q$ whose coefficient sequences
are \emph{non-unimodal}, thereby disproving a conjecture by Arratia~et~al.
We have verified that for graphs of order up to 12, all polynomials $Q$ have unimodal coefficients.
It has been shown that LC and ELC orbits of graphs correspond to 
equivalence classes of certain \emph{error-correcting codes} and \emph{quantum states}.
We show that the properties of these codes and quantum states are related to properties
of the associated interlace polynomials.
\end{abstract}

\section{Introduction}

A \emph{graph} is a pair $G=(V,E)$ where $V$ is a set of \emph{vertices},
and $E \subseteq V \times V$ is a set of \emph{edges}. The \emph{order} of $G$
is $n = |V|$.
We will only consider \emph{simple} \emph{undirected} graphs, i.e., 
graphs where all edges are bidirectional and no vertex can be adjacent to itself.
The \emph{neighbourhood} of $v \in V$, denoted $N_v \subset V$, is the set of 
vertices connected to $v$ by an edge. 
The number of vertices adjacent to $v$ is called the \emph{degree} of $v$.
An \emph{Eulerian graph} is a graph where all vertices have even degree.
The \emph{induced subgraph} of $G$ on $W \subseteq V$ 
contains vertices $W$ and all edges from $E$ whose endpoints are both in $W$.
The \emph{complement} of $G$ is found by replacing $E$ with $V \times V - E$,
i.e., the edges in $E$ are changed to non-edges, and the non-edges to edges.
Two graphs $G=(V,E)$ and $G'=(V,E')$ are \emph{isomorphic} if and only if
there exists a permutation $\pi$ on $V$ such that $\{u,v\} \in E$ if and 
only if $\{\pi(u), \pi(v)\} \in E'$.
A \emph{path} is a sequence of vertices, $(v_1,v_2,\ldots,v_i)$, such that
$\{v_1,v_2\}, \{v_2,v_3\},$ $\ldots, \{v_{i-1},v_{i}\} \in E$.
A graph is \emph{connected} if there is a path from any vertex to any other vertex in the graph.
A graph is \emph{bipartite} if its set of vertices can be decomposed into two disjoint sets 
such that no two vertices within the same set are adjacent.
A \emph{complete graph} is a graph where all pairs of vertices are
connected by an edge. A \emph{clique} is a complete subgraph.
A $k$-clique is a clique consisting of $k$ vertices.
An \emph{independent set} is the complement of a clique, i.e., an empty subgraph.
The \emph{independence number} of $G$ is the size of the largest independent set in $G$.

\begin{defn}[\hspace{1pt}\hspace{-1pt}\cite{flaas,bouchet,hubert}]
Given a graph $G=(V,E)$ and a vertex $v \in V$,
let $N_v \subset V$ be the neighbourhood of $v$.
\emph{Local complementation} (LC) on $v$ transforms $G$ into $G * v$ by
replacing the induced subgraph of $G$ on $N_v$ by its complement. (Fig.~\ref{fig:lcexample})
\end{defn}

\begin{figure}
 \centering
 \subfloat[The Graph $G$]
 {\hspace{5pt}\includegraphics[width=.28\linewidth]{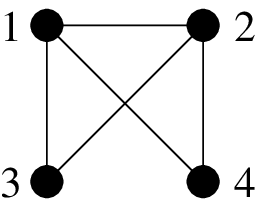}\hspace{5pt}\label{fig:lcexample1}}
 \quad
 \subfloat[The Graph $G*1$]
 {\hspace{5pt}\includegraphics[width=.28\linewidth]{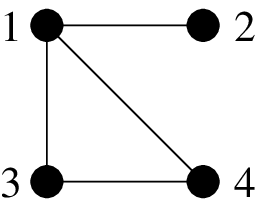}\hspace{5pt}\label{fig:lcexample2}}
 \caption{Example of Local Complementation}\label{fig:lcexample}
\end{figure}

\begin{defn}[\hspace{1pt}\hspace{-1pt}\cite{bouchet}]\label{prop:triplelc}
Given a graph $G=(V,E)$ and an edge $\{u,v\} \in E$, \emph{edge local complementation} (ELC)
on $\{u,v\}$ transforms $G$ into $G^{(uv)} = G*u*v*u = G*v*u*v$.
\end{defn}

\begin{defn}[\hspace{1pt}\hspace{-1pt}\cite{bouchet}]\label{def:elc}
ELC on $\{u,v\}$ can equivalently be defined as follows. Decompose $V\setminus \{u,v\}$ 
into the following four disjoint sets, as visualized in Fig.~\ref{fig:elc}.
\renewcommand{\labelenumi}{$\Alph{enumi}$}
\begin{enumerate}
\item Vertices adjacent to $u$, but not to $v$.
\item Vertices adjacent to $v$, but not to $u$.
\item Vertices adjacent to both $u$ and $v$.
\item Vertices adjacent to neither $u$ nor $v$.
\end{enumerate}
To obtain $G^{(uv)}$, perform the following procedure.
For any pair of vertices $\{x,y\}$, where $x$ belongs to class $A$, $B$, or $C$,
and $y$ belongs to a different class $A$, $B$, or $C$, ``toggle'' the pair $\{x, y\}$,
i.e., if $\{x,y\} \in E$, delete the edge, and if $\{x,y\} \not\in E$, add the edge
$\{x,y\}$ to $E$. Finally, swap the labels of vertices $u$ and $v$.
\end{defn}

\begin{figure}
 \centering
 \includegraphics[width=.43\linewidth]{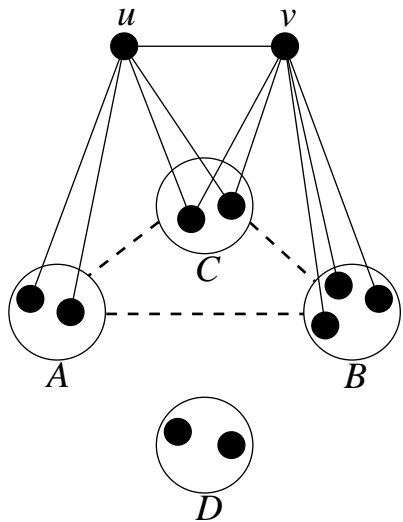}
 \caption{Visualization of the ELC Operation}\label{fig:elc}
\end{figure}

\begin{defn}
The \emph{LC orbit} of a graph $G$ is the set of all unlabeled graphs that
can be obtained by performing any sequence of LC operations on~$G$.
Similarly, the \emph{ELC orbit} of $G$ comprises all unlabeled graphs that
can be obtained by performing any sequence of ELC operations on~$G$.
\end{defn}

The LC operation was first defined by de Fraysseix~\cite{hubert}, and later studied
by Fon-der-Flaas~\cite{flaas} and Bouchet~\cite{bouchet}. Bouchet defined ELC 
as ``complementation along an edge''~\cite{bouchet}, but this operation is also 
known as \emph{pivoting} on a graph~\cite{interlace}.

The recently defined \emph{interlace polynomials} are based on the LC and ELC operations. 
Arratia, Bollobás, and Sorkin~\cite{interlace} defined the interlace polynomial $q(G)$ of 
the graph $G$. This work was motivated by a problem related to DNA sequencing~\cite{eulerdna}.

\begin{defn}[\hspace{1pt}\hspace{-1pt}\cite{interlace}]\label{def:q}
For every graph $G$, there is an associated interlace polynomial $q(G,x)$,
which we will usually denote $q(G)$ for brevity.
For the edgeless graph of order $n$, $E_n=(V,\emptyset)$, $q(E_n) = x^n$.
For any other graph $G=(V,E)$, choose an arbitrary edge $\{u,v\} \in E$, and let
\[
q(G) = q(G \setminus u) + q(G^{(uv)} \setminus u),
\]
where $G \setminus u$ is the graph $G$ with vertex $u$ and all edges incident on $u$ removed.
\end{defn}

It was proven by Arratia~et~al.~\cite{interlace} that the polynomial is independent
of the order of removal of edges, and that the polynomial is invariant under ELC, i.e.,
that $q(G) = q(G^{(uv)})$ for any edge $\{u,v\}$. 

Aigner and van der Holst~\cite{aigner} later defined the interlace polynomial
$Q(G)$ which similarly encodes properties of the LC orbit of $G$.

\begin{defn}[\hspace{1pt}\hspace{-1pt}\cite{aigner}]\label{def:Q}
For every graph $G$, there is an associated interlace polynomial $Q(G,x)$,
which we will usually denote $Q(G)$ for brevity.
For the edgeless graph of order $n$, 
$E_n=(V,\emptyset)$, $Q(E_n) = x^n$.
For any other graph $G=(V,E)$, choose an arbitrary edge $\{u,v\} \in E$, and let
\[
Q(G) = Q(G \setminus u) + Q(G^{(uv)} \setminus u) + Q(G*u \setminus u).
\]
\end{defn}

Again, the order of removal of edges is irrelevant, and the polynomial is invariant under LC and ELC.
It was shown by Aigner and van der Holst~\cite{aigner} that both $q(G)$ and $Q(G)$ can also
be derived from the ranks of matrices obtained by certain modifications of the adjacency matrix of $G$.
A similar approach, but expressed in terms of certain sets of \emph{local unitary transforms},
was shown by Riera and Parker~\cite{genbent}. 
If $G$ is an unconnected graph with components $G_1$ and $G_2$, 
then $q(G) = q(G_1)q(G_2)$ and $Q(G) = Q(G_1)Q(G_2)$.

The interlace polynomials $q(G)$ and $Q(G)$ summarize several properties of the ELC and LC orbits
of the graph $G$. The degree of the lowest-degree term of $q(G)$
equals the number of connected components of $G$, and is therefore one for a connected graph~\cite{interlace}.
The degree of $q(G)$ equals the maximum independence number over all graphs in the ELC orbit of $G$~\cite{aigner}. 
It follows that the degree of $q(G)$ is also an upper bound on the independence number of $G$.
Likewise, the degree of $Q(G)$ gives the size of the largest independent set 
in the LC orbit of~$G$~\cite{setapaper}.
The degree of $Q(G)$ will always be greater than or equal to the degree of $q(G)$.
Evaluating interlace polynomials for certain values of $x$ can also gives us 
some information about the associated graphs.
For a graph $G$ of order $n$, it always holds that $q(G,2) = 2^n$ and $Q(G,3) = 3^n$.
$q(G,1)$ equals the number of induced subgraphs of $G$ with an odd number of \emph{perfect matchings}~\cite{aigner}.
$Q(G,2)$ equals the number of general induced subgraphs of $G$ (with possible loops attached to the vertices)
with an odd number of general perfect matchings~\cite{aigner}.
$Q(G,4)$ equals $2^n$ times the number of induced Eulerian subgraphs of $G$~\cite{aigner}.
It has been shown that $q(G,-1) = (-1)^n 2^{n-r}$, where $n$ is the order of $G$
and $r$ is the rank over $\mathbb{F}_2$ of $\Gamma + I$, where $\Gamma$ is the adjacency 
matrix of $G$~\cite{aigner,atminusone}.
$q(G,3)$ is always divisible by $q(G,1)$, and the quotient is an odd integer~\cite{aigner}.

\begin{figure}
 \centering
 {
 \subfloat[]{\hspace{5pt}\includegraphics[width=.35\linewidth]{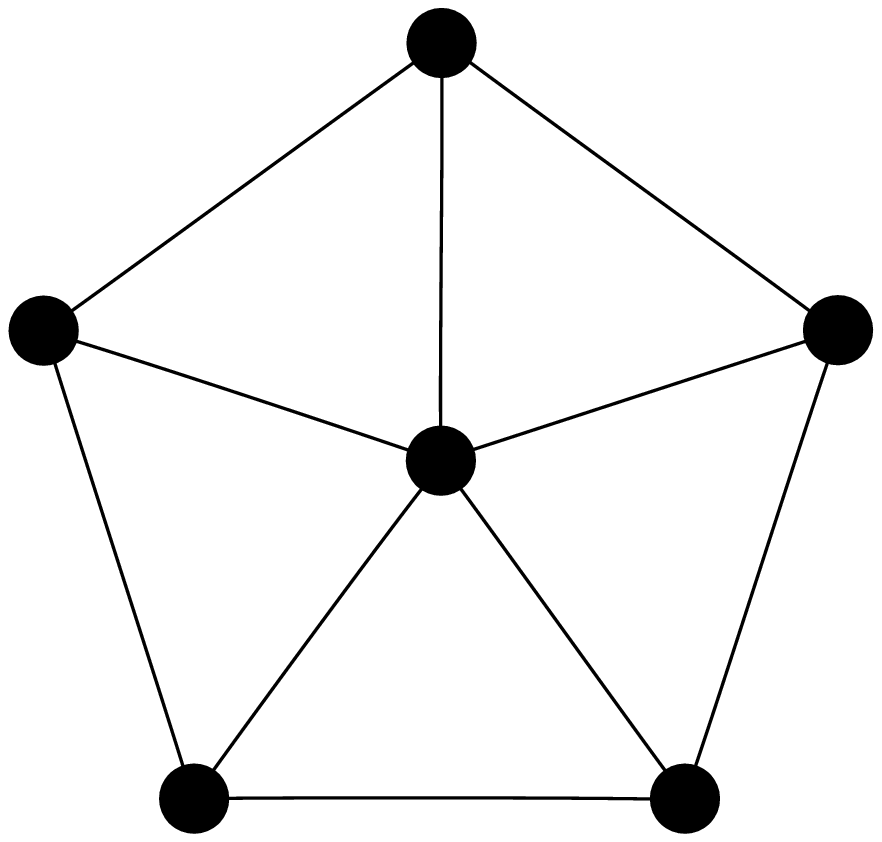}\label{fig:wheel}\hspace{5pt}}
 \quad
 \subfloat[]{\hspace{5pt}\includegraphics[width=.35\linewidth]{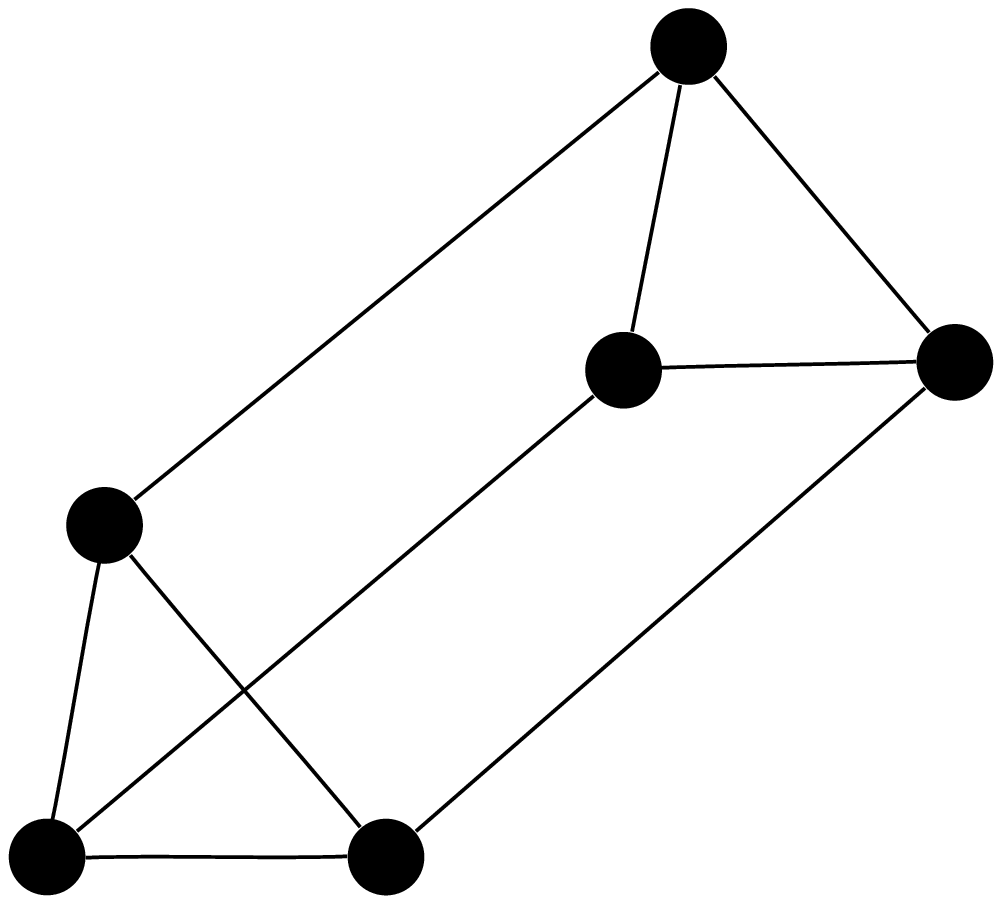}\label{fig:doubleclique}\hspace{5pt}}
 }
 \caption{Example of an LC Orbit}\label{fig:hexacodes}
\end{figure}

\begin{exmp}
The two graphs in Fig.~\ref{fig:hexacodes} comprise an LC orbit, and an ELC orbit.
(Note that in general, an LC orbit can be decomposed into one or more ELC orbits.)
Both graphs have interlace polynomials
$q(G) = 12x + 10x^2$ and $Q(G) = 108x + 45x^2$. The fact that $\deg(Q) = 2$ matches the observation
that none of the two graphs have an independent set of size greater than two.
That $\frac{Q(G,4)}{2^6} = 18$ means that each graph has 18 Eulerian subgraphs.
\end{exmp}

In their list of open problems~\cite{interlace}, Arratia~et~al. pose the question
of how many different interlace polynomials there are for graphs of order $n$.
In Section~\ref{sec:class}, 
we answer this question for $n \le 12$, for both interlace polynomials $q$ and $Q$.

In the DNA sequencing setting~\cite{eulerdna}, interlace polynomials of \emph{circle graphs}
are of particular interest. Arratia~et~al.~\cite{eulerdna} enumerated the circle graphs of order
up to 9. In Section~\ref{sec:circle}, we extend this enumeration to order 12.

Let $q(G) = a_1x + a_2x^2 + \cdots + a_d x^d$. Then the sequence of
coefficients of $q$ is $\{a_i\} = (a_1, a_2, \ldots, a_d)$.
Arratia~et~al.~\cite{interlace} conjecture that this sequence is \emph{unimodal}
for all $q$. The sequence $\{a_i\}$ is unimodal if, for some $1 \le k \le d$,
$a_i \le a_j$ for all $i < j \le k$, and $a_i \ge a_j$ for all $i > j \ge k$.
In other words, the sequence is non-decreasing up to some coefficient $k$,
and the rest of the sequence is non-increasing.
In Section~\ref{sec:nonuni}, we show that there exist interlace polynomials $q$ whose 
coefficient sequences are non-unimodal, and thereby disprove the conjecture by Arratia~et~al.
Our enumeration shows that all interlace polynomials of graphs of order up to 9 are unimodal, but
that there are two graphs of order 10 with the same non-unimodal interlace polynomial.
From these graphs of order 10 it is possible 
to construct graphs of any order greater than 10 with non-unimodal interlace polynomials.
We verify that all interlace polynomials $Q(G)$ and all polynomials $x \cdot q(G,x+1)$ of graphs
of order up to 12 have unimodal coefficients.

In Section~\ref{sec:codes} we highlight an interesting relationship between
interlace polynomials, error-correcting codes, and quantum states.
The LC orbit of a graph corresponds to the equivalence class of a 
\emph{self-dual quantum code}~\cite{nest},
and ELC orbits of bipartite graphs correspond to  
equivalence classes of \emph{binary linear codes}~\cite{pivotlinear}.
In both cases, the \emph{minimum distance} of a code is given by $\delta + 1$, where $\delta$ is the
minimum vertex degree over all graphs in the corresponding orbit.
We have previously shown~\cite{setapaper} that a self-dual quantum code with high minimum distance
often corresponds to a graph $G$ where $\deg(Q)$, the degree of $Q(G)$, is small.
A self-dual quantum code can also be interpreted as 
a \emph{quantum graph state}~\cite{hein}. A code with high minimum distance will
correspond to a quantum state with a high degree of \emph{entanglement}.
The degree of $Q(G)$ gives an indicator of the entanglement in the graph state represented
by $G$ known as the \emph{peak-to-average power ratio}~\cite{setapaper} 
with respect to certain transforms. Another indicator of the 
entanglement in a graph state is the \emph{Clifford merit factor} (CMF)~\cite{cmf},
which can be derived from the evaluation of $Q(G)$ at $x=4$~\cite{interlacespectral}.
In Section~\ref{sec:codes} we give the range of possible values of $\delta$, $\deg(Q)$, and $Q(G,4)$
for graphs of order up to 12, and bounds on these parameters for graphs of order up to 25,
derived from the best known self-dual quantum codes.

\section{Enumeration of Interlace Polynomials}\label{sec:class}

In the context of error-correcting codes, we have previously classified the 
LC orbits~\cite{selfdual} and ELC orbits~\cite{pivotboolean,pivotlinear} 
of all graphs on up to 12 vertices.
In Table~\ref{tab:lcorbits}, the sequence $\{c_{L,n}\}$ gives the number of LC orbits 
of connected graphs on $n$ vertices, while $\{t_{L,n}\}$ gives the total number of 
LC orbits of graphs on $n$ vertices. Similarly, the sequence $\{c_{E,n}\}$ gives the number of ELC orbits 
of connected graphs on $n$ vertices, while $\{t_{E,n}\}$ gives the total number of 
ELC orbits of graphs on $n$ vertices. A database containing one representative from each LC orbit is available 
at \url{http://www.ii.uib.no/~larsed/vncorbits/}. A similar database of ELC orbits
can be found at \url{http://www.ii.uib.no/~larsed/pivot/}.

\begin{table}
\centering
\caption{Number of LC and ELC Orbits}
\label{tab:lcorbits}
\begin{tabular}{rrrrr}
\toprule
$n$ & $c_{L,n}$ & $t_{L,n}$ & $c_{E,n}$ & $t_{E,n}$ \\
\midrule
1  &1         &1         &1           &1           \\
2  &1         &2         &1           &2           \\
3  &1         &3         &2           &4           \\
4  &2         &6         &4           &9           \\
5  &4         &11        &10          &21          \\
6  &11        &26        &35          &64          \\
7  &26        &59        &134         &218         \\
8  &101       &182       &777         &1068       \\
9  &440       &675       &6702        &8038       \\
10 &3132      &3990      &104,825     &114,188     \\
11 &40,457    &45,144    &3,370,317   &3,493,965   \\
12 &1,274,068 &1,323,363 &231,557,290 &235,176,097 \\
\bottomrule
\end{tabular}
\end{table}

Note that the value of $t_n$ (for either ELC or LC orbits)
can be derived easily once the sequence $\{c_m\}$ is known 
for $1 \le m \le n$, using the \emph{Euler transform}~\cite{sloane2},
\begin{eqnarray*}
a_n &=& \sum_{d|n} d c_d,\\
t_1 &=& a_1,\\
t_n &=& \frac{1}{n}\left( a_n + \sum_{k=1}^{n-1} a_k t_{n-k} \right).
\end{eqnarray*}

The question of how many distinct interlace polynomials there are for graphs of order $n$
was posed by Arratia~et~al.~\cite{interlace}. For a representative from each LC and ELC orbit,
we have calculated the interlace polynomials $Q$ and $q$, respectively. We then counted the number
of distinct interlace polynomials.
In Table~\ref{tab:polys}, the sequence $\{c_{Q,n}\}$ gives the number of interlace polynomials $Q$
of connected graphs of order $n$, while $\{t_{Q,n}\}$ gives the total number of 
interlace polynomials $Q$ of graphs of order $n$. Similarly, $\{c_{q,n}\}$ and $\{t_{q,n}\}$
give the numbers of interlace polynomials $q$.
We observe that in Table~\ref{tab:polys}, the relationship $t_{q,n} = c_{q,n} + t_{q,n-1}$ holds.

\begin{table}
\centering
\caption{Number of Distinct Interlace Polynomials}
\label{tab:polys}
\begin{tabular}{rrrrr}
\toprule
$n$ & $c_{Q,n}$ & $t_{Q,n}$ & $c_{q,n}$ & $t_{q,n}$ \\
\midrule
1  &        1 &        1 &          1 &          1 \\
2  &        1 &        2 &          1 &          2 \\
3  &        1 &        3 &          2 &          4 \\
4  &        2 &        6 &          4 &          8 \\
5  &        4 &       11 &          9 &         17 \\
6  &       10 &       24 &         24 &         41 \\
7  &       23 &       52 &         71 &        112 \\
8  &       84 &      152 &        257 &        369 \\
9  &      337 &      521 &       1186 &       1555 \\
10 &     2154 &     2793 &       7070 &       8625 \\
11 &   22,956 &   26,178 &     56,698 &     65,323 \\
12 &  486,488 &  515,131 &    614,952 &    680,275 \\
\bottomrule
\end{tabular}
\end{table}

\section{Enumeration of Circle Graphs}\label{sec:circle}

A graph $G$ is a \emph{circle graph} if each vertex in $G$ can be represented as a chord
on a circle, such that two chords intersect if and only if there is an edge between
the two corresponding vertices in $G$. An example of a circle graph and its
corresponding circle diagram is given in Fig.~\ref{fig:circleexample}.

\begin{figure}
 \centering
 \subfloat[The Circle Graph $G$]
 {\hspace{5pt}\includegraphics[height=.2\linewidth]{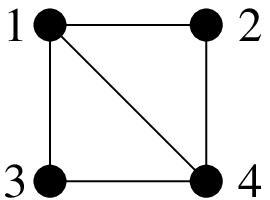}\hspace{5pt}\label{fig:circleexample1}}
 \quad
 \subfloat[The Circle Representation of $G$]
 {\hspace{5pt}\includegraphics[height=.3\linewidth]{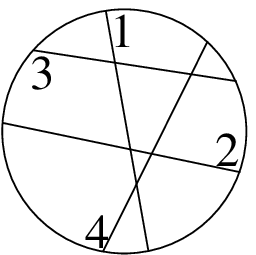}\hspace{5pt}\label{fig:circleexample2}}
 \caption{Example of a Circle Graph}\label{fig:circleexample}
\end{figure}

\begin{figure}
 \centering
 \includegraphics[height=.25\linewidth]{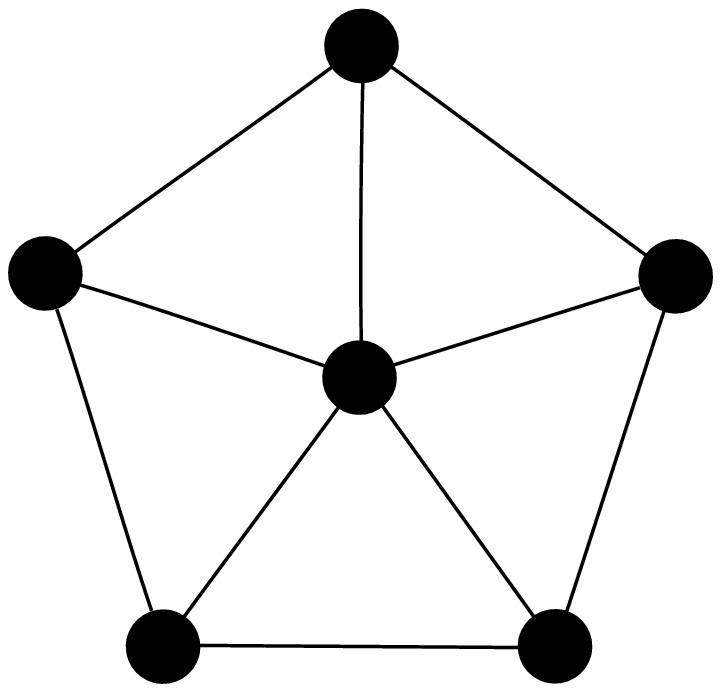}
 \quad
 \includegraphics[height=.25\linewidth]{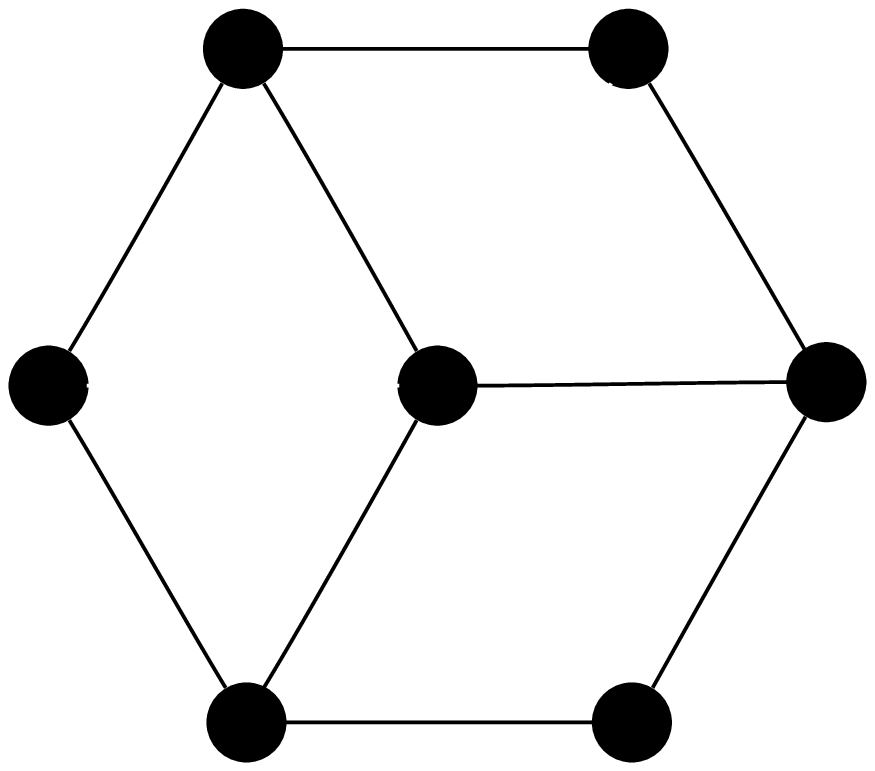}
 \quad
 \includegraphics[height=.25\linewidth]{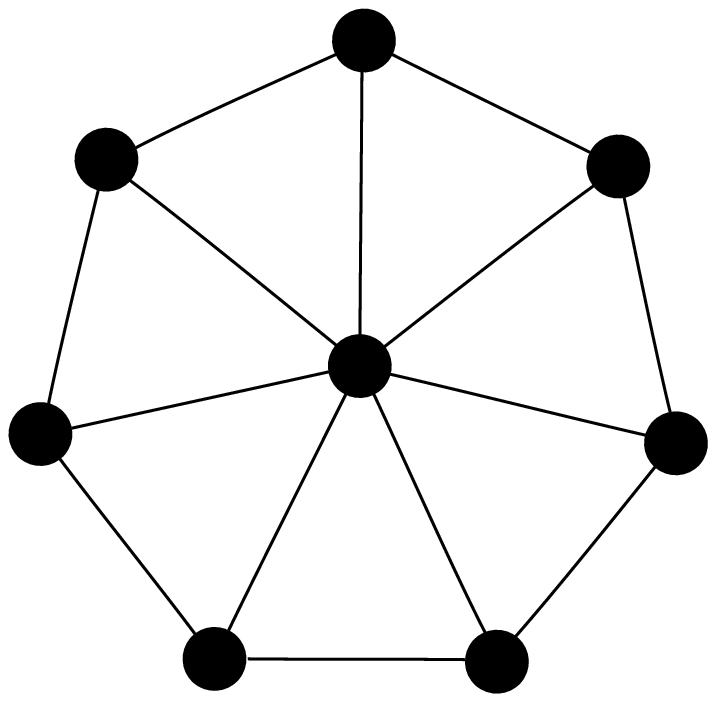}
 \caption{Circle Graph Obstructions}\label{fig:obstructions}
\end{figure}

Whether a given graph is a circle graph can be recognized in polynomial time~\cite{spinrad}. 
It is also known that LC operations will map a circle graph to a circle graph, 
and a non-circle graph to a non-circle graph~\cite{bouchetcircle}. Bouchet~\cite{bouchetcircle}
proved that a graph $G$ is a circle graph if and only if certain \emph{obstructions}, shown
in Fig.~\ref{fig:obstructions}, do not appear as subgraphs anywhere in the LC orbit of $G$.

Arratia~et~al.~\cite{eulerdna} pointed out that an enumeration of circle graphs
did not seem to have appeared in the literature before, and then gave an enumeration
of circle graphs of order up to 9. Using our previous classification of LC orbits, and
the fact that the circle graph property is preserved by LC operations, we are able to
generate all circle graphs of order up to 12.
In Table~\ref{tab:circle}, the sequence $\{c_{c,n}\}$ gives the number of connected circle graphs
of order $n$, while $\{t_{c,n}\}$ gives the total number of circle graphs of order $n$.
The sequences $\{c_{c,n}'\}$ and $\{t_{c,n}'\}$ give the number of LC orbits containing
circle graphs.
A database with one representative from each LC orbit of connected circle graphs is available 
at \url{http://www.ii.uib.no/~larsed/circle/}.

\begin{table}
\centering
\caption{Number of Circle Graphs}
\label{tab:circle}
\begin{tabular}{rrrrr}
\toprule
$n$ & $c_{c,n}$ & $t_{c,n}$ & $c_{c,n}'$ & $t_{c,n}'$ \\
\midrule
1  &           1 &           1 &           1 &           1 \\
2  &           1 &           2 &           1 &           2 \\
3  &           2 &           4 &           1 &           3 \\
4  &           6 &          11 &           2 &           6 \\
5  &          21 &          34 &           4 &          11 \\
6  &         110 &         154 &          10 &          25 \\
7  &         789 &         978 &          23 &          55 \\
8  &        8336 &        9497 &          81 &         157 \\
9  &     117,283 &     127,954 &         293 &         499 \\
10 &   2,026,331 &   2,165,291 &        1403 &        2059 \\
11 &  40,302,425 &  42,609,994 &        7968 &      10,543 \\
12 & 892,278,076 & 937,233,306 &      55,553 &      68,281 \\
\bottomrule
\end{tabular}
\end{table}

\section{Unimodality}\label{sec:nonuni}

Having calculated the interlace polynomials $q$ of all graphs of order up to 12, it
was possible to check whether their coefficient sequences were unimodal, as 
conjectured by Arratia~et~al.~\cite{interlace}. Note that a similar conjecture
has been disproved for the related \emph{Tutte polynomial}~\cite{unitutte}.

Our results show that all interlace polynomials $q$ of graphs of order $n \le 9$ are
unimodal, but that for $n=10$ there exists a single non-unimodal interlace polynomial with
coefficient sequence $\{a_i\} = (2,7,6,7,4,3,2,1,0,0)$.
Only two graphs on 10 vertices, comprising a single ELC orbit, correspond to this polynomial.
One of these graphs is shown in Fig.~\ref{fig:nonunimodal}.

\begin{figure}
 \centering
 \includegraphics[width=.40\linewidth]{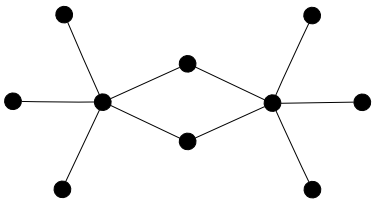}
 \caption{The Smallest Graph with Non-Unimodal Interlace Polynomial $q$}\label{fig:nonunimodal}
\end{figure}

We have further found that, up to ELC equivalence, there are 4 graphs on 11 vertices with
non-unimodal interlace polynomials, 3 of which are connected graphs, and 20 graphs
on 12 vertices with non-unimodal polynomials, 15 of which are connected.

Given the single non-unimodal interlace polynomial of a graph of order $n=10$, it is easy to show that
there must exist non-unimodal interlace polynomials for all $n>10$,
since the following methods of extending a graph will preserve the non-unimodality of the associated
interlace polynomial. Given a graph $G$ on $n$ vertices with non-unimodal interlace polynomial, 
we can add an isolated vertex to obtaining an unconnected graph $G'$ on $n+1$ vertices, 
where $q(G') = x \cdot q(G)$ is clearly also non-unimodal. Non-unimodality
is also preserved by \emph{substituting} a vertex $v$ of $G$ by a clique of size $m$, i.e.,
we obtain $G'$ where $v$ is replaced by $m$ vertices, all connected to each other and all 
connected to $w$ whenever $\{v,w\}$ is an edge in $G$. It can then be shown that $q(G') =
2^m q(G)$~\cite[Prop.~38]{interlace}.

\begin{prop}~\label{prop:duplicate}
Given a graph $G$, let $G'$ be the graph obtained by \emph{duplicating} a
vertex $v$ of $G$, i.e., by adding a vertex $v'$ such that $v'$ is connected to $w$ whenever 
$\{v,w\}$ is an edge in $G$.
The interlace polynomial of $G$ can be written $q(G) = a(x) + cx^j + x^{j+1} b(x)$, where
$a$ and $b$ are arbitrary polynomials, $c$ is a constant, and $j = \deg(a) + 1$.
The unimodality or lack thereof of $G$ will be preserved in $G'$ if $q(G \setminus v) = a(x) + x^j b(x)$.
\end{prop}
\begin{proof}
By duplicating the vertex $v$, we obtain a graph $G'$ with interlace polynomial 
$q(G') = (1+x)q(G) - x \cdot q(G \setminus v)$~\cite[Prop.~40]{interlace}.
If the condition above is satisfied, $q(G') = x^{j+2} a(x) + c(x^{j+1} + x^j) + b(x)$.
The only difference between the coefficient sequences of $q(G)$ and $q(G')$ is that
the coefficient $c$ is repeated in $q(G')$, and unimodality or non-unimodality must therefore
be preserved.
\end{proof}

Let $G$ be the graph depicted in Fig.~\ref{fig:nonunimodal}, and let $v$ be one of 
the six vertices of degree one in this graph. If we duplicate $v$ we obtain a graph
whose interlace polynomial has the non-unimodal coefficient sequence $(2,7,6,7,6,4,3,2,1,0,0)$.
According to Prop.~\ref{prop:duplicate}, we can repeat the duplication of a vertex 
with degree one and the coefficient sequence will remain $(2,7,6,7,6,\ldots,6,4,3,2,1,0,0)$, 
i.e., non-unimodal. 

By the described extension methods we can obtain, from the single graph on 10 vertices
shown in Fig.~\ref{fig:nonunimodal}, all the 4 inequivalent graphs on 11 vertices
and 16 of the 20 inequivalent graphs on 12 vertices with non-unimodal interlace polynomials.
Representatives from the ELC orbits of the 4 non-trivial graphs on 12 vertices 
with non-unimodal interlace polynomials are shown in Fig.~\ref{fig:nonunimodal12}.

\begin{figure}
 \centering
 \includegraphics[height=.26\linewidth]{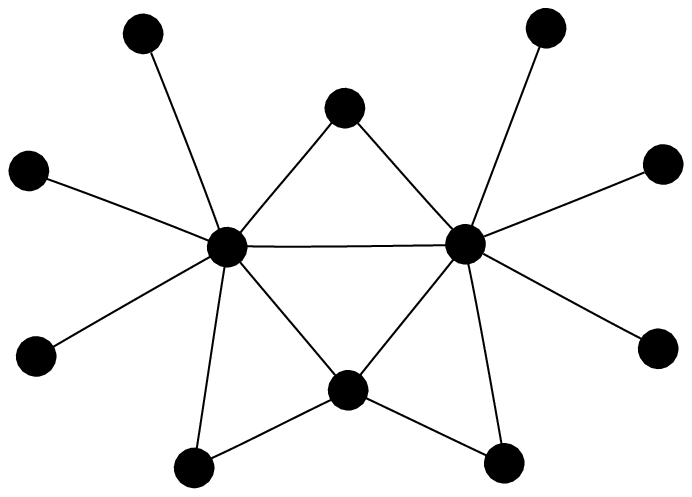}
 \quad
 \includegraphics[height=.26\linewidth]{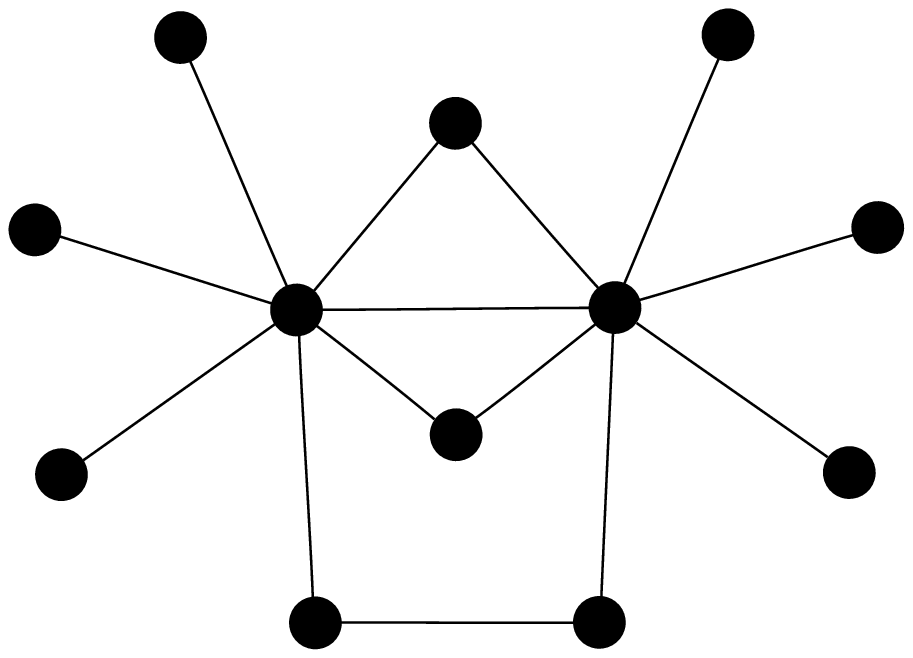}\\\vspace{10pt}
 \includegraphics[height=.24\linewidth]{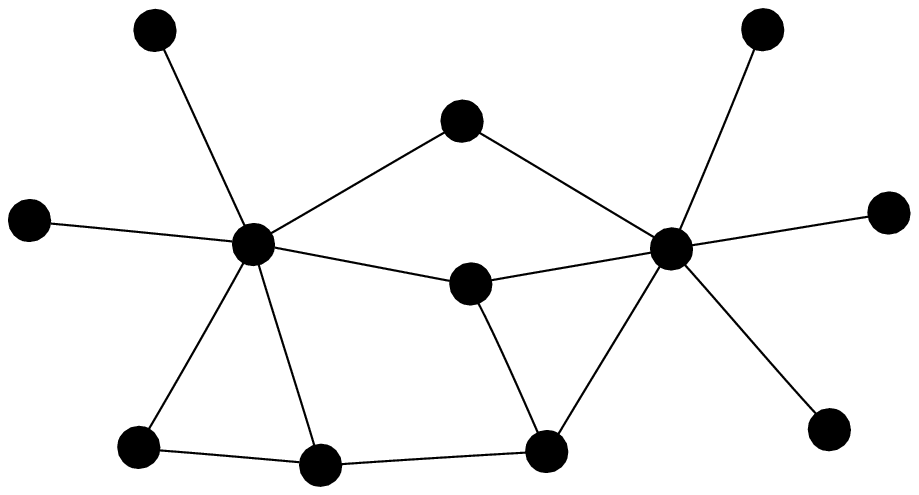}
 \quad
 \includegraphics[height=.24\linewidth]{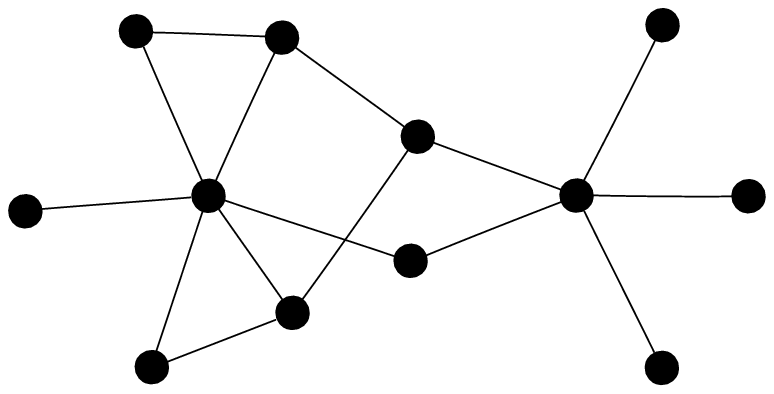}
 \caption{Non-trivial Graphs of Order 12 with Non-Unimodal Interlace Polynomial $q$}\label{fig:nonunimodal12}
\end{figure}

The two following conjectures have been checked for all graphs on up to 12 vertices,
and no counterexamples have been found.

\begin{con}[\hspace{1pt}\hspace{-1pt}\cite{interlace}]
For any interlace polynomial $q(G,x)$, the associated polynomial $x \cdot q(G,x+1)$ has a unimodal coefficient
sequence.
\end{con}

\begin{con}
For any graph $G$, the interlace polynomial $Q(G)$ has a unimodal coefficient sequence.
\end{con}

\section{Connections to Codes and Quantum States}\label{sec:codes}

An important question is what the interlace polynomials $q(G)$ and $Q(G)$ actually compute
about the graph $G$ itself. When $G$ is a circle graph, $q(G)$ can be used to solve counting problems
relevant to DNA sequencing~\cite{eulerdna}. We will show that the interlace polynomials also give
clues about the error-correction capability of codes and the entanglement of quantum states.

It is known that \emph{self-dual quantum codes}, so called because they
correspond to self-dual additive codes over $\mathbb{F}_4$~\cite{calderbank}, 
can be represented as graphs~\cite{schlingemann,grasslpaper}.
The LC orbit of a graph corresponds to the equivalence class of a 
self-dual quantum code~\cite{nest}.
Similarly, the ELC orbits of bipartite graphs correspond to 
equivalence classes of \emph{binary linear codes}~\cite{pivotlinear}.
In both cases the \emph{minimum distance}, an important parameter that determines the error-correcting
capability of a code, is given by $\delta + 1$, where $\delta$ is the minimum vertex degree over 
all graphs in the corresponding LC or ELC orbit. A self-dual quantum code can also be interpreted
as a \emph{quantum graph state}~\cite{hein}, and the $\delta$-value of the associated LC orbit is then 
an indicator of the degree of \emph{entanglement} in the state.

Although the value $\delta$ can not be obtained from an interlace polynomial, several values that
are correlated with $\delta$ are encoded in the interlace polynomial.
The size of the largest independent set over 
all members of the LC orbit of $G$ equals $\deg(Q)$, the degree of $Q(G)$~\cite{aigner,setapaper}.
We have previously shown that optimal self-dual quantum codes
correspond to LC orbits where $\deg(Q)$ is small~\cite{setapaper}.
These codes have largest possible minimum distance for a given length $n$, and thus
the associated LC orbits of graphs on $n$ vertices have maximum possible values of $\delta$.
The data in Table~\ref{tab:degQ} implies that the LC orbits with the highest $\delta$-values also have 
the lowest values of $\deg(Q)$, but that the converse is not always true.
In the context of quantum graph states, the value $2^{\deg(Q)}$ is equal 
to the \emph{peak-to-average power ratio}~\cite{setapaper} with respect to certain transforms, 
which is another indicator of the degree of entanglement in the state.

\begin{table}
\centering
\caption{Range of $\deg(Q)$ For Given $\delta$ and $n$}
\label{tab:degQ}
\begin{tabular}{cccccccccccccc}
\toprule
$\delta \backslash n$ & 2 & 3 & 4 &  5      & 6      & 7     & 8      & 9      & 10      & 11    & 12 \\ 
\midrule
1 & 1 & 2 & 2,3 & 3,4    & 3--5   & 3--6  & 3--7   & 4--8   & 4--9    & 4--10 & 4--11\\
2 &   &   &     & 2      & 3      & 3,4   & 3,4    & 3--5   & 4--6    & 4--7  & 4--8 \\
3 &   &   &     &        & 2      &       & 3,4    & 3,4    & 3--5    & 4--6  & 4--7 \\
4 &   &   &     &        &        &       &        &        &         & 4     & 4  \\
5 &   &   &     &        &        &       &        &        &         &       & 4  \\
\bottomrule
\end{tabular}
\end{table}

Another measure of the entanglement in a quantum graph state is the 
\emph{Clifford merit factor} (CMF)~\cite{cmf}.
A quantum graph state can be represented as a graph $G$, and the CMF of the state can be derived 
from the value obtained by evaluating the associated interlace 
polynomial $Q(G)$ at $x=4$~\cite{interlacespectral}.
The CMF value can be obtained with the formula $\frac{6^n}{2^n Q(G,4) - 6^n}$.
Interestingly, $\frac{Q(G,4)}{2^n}$ is also the number of induced Eulerian subgraphs of a graph 
on $n$ vertices~\cite{aigner}, which is invariant over the LC orbit.
As can be seen in Table~\ref{tab:Q4}, the LC orbits with the highest $\delta$-values also have 
the lowest values of $Q(G,4)$.
Other evaluations of the interlace polynomials are also of interest in the context of quantum graph
states, for instance $q(G,1)$ and $Q(G,2)$ give the number of \emph{flat spectra} with respect to certain
sets of transforms of the state~\cite{interlacespectral}.

\begin{table}
\centering
\caption{Range of $\frac{Q(G,4)}{2^n}$ For Given $n$ and $\delta$}
\label{tab:Q4}
\begin{tabular}{cccccc}
\toprule
$n \backslash \delta$ & 1 & 2 & 3 & 4 & 5 \\ 
\midrule
2  & 3         &          &          &          &    \\
3  & 5         &          &          &          &    \\
4  & 8--9      &          &          &          &    \\
5  & 13--17    & 12       &          &          &    \\
6  & 20--33    & 19       & 18       &          &    \\
7  & 30--65    & 29--30   &          &          &    \\
8  & 47--129   & 45--48   & 44--45   &          &    \\
9  & 73--257   & 69--80   & 68--69   &          &    \\
10 & 112--513  & 106--128 & 104--109 &          &    \\
11 & 172--1025 & 160--183 & 157--180 & 156      &    \\
12 & 260--2049 & 244--362 & 237--288 & 238--239 & 234 \\
\bottomrule
\end{tabular}
\end{table}

Although no algorithm is known for computing the interlace polynomial of a graph efficiently,
it is in general faster to generate interlace polynomials, by simply using the recursive
algorithm given in Definitions~\ref{def:q} and~\ref{def:Q}, than it is to generate the entire 
ELC or LC orbits of a graph.
Note that calculating $\delta$ can also be done faster than by generating the complete LC orbit, 
by using methods for calculating the minimum distance of a self-dual quantum code represented
as a graph~\cite{selfdual}.
We have calculated the interlace polynomials $Q$ of graphs corresponding to the best known self-dual
quantum codes, obtained from \url{http://www.codetables.de/} and from a search we have previously performed
of \emph{circulant graph codes}~\cite{nonbinary}.
An adjacency matrix is called \emph{circulant} if the $i$-th row is equal to the
first row, cyclically shifted $i-1$ times to the right.
The results, for graphs of order $n$ up to 25, are given in Table~\ref{tab:bounds}.
Values printed in bold font are the best values we have found, and are thus upper bounds
on the minimum possible values of $\deg(Q)$ and $Q(G,4)$ for the given $n$.
The values of $\delta$ printed in bold font are known to be optimal, except for $n=23$ and $n=25$, where a graph with
$\delta=8$ could exist, and $n=24$, $n=26$, and $n=27$, where $\delta=9$ is possible. In general,
the following bounds hold~\cite{calderbank}.
\[
\delta \le 2 \left\lfloor \frac{n}{6} \right\rfloor + 1,
\]
if the corresponding self-dual quantum code is of \emph{Type~II}, which means that its graph
representation is \emph{anti-Eulerian}~\cite{selfdual}, i.e., a graph where all vertices have odd degree. Such
graphs must have an even number of vertices, and it is interesting to note that the anti-Eulerian
property is preserved by LC operations.
\[
\delta \le \begin{cases}
2 \left\lfloor \frac{n}{6} \right\rfloor, \quad \text{if } n \equiv 0 \text{ } (\text{mod } 6) \\ 
2 \left\lfloor \frac{n}{6} \right\rfloor + 2, \quad \text{if } n \equiv 5 \text{ } (\text{mod } 6) \\ 
2 \left\lfloor \frac{n}{6} \right\rfloor + 1, \quad \text{otherwise,}
\end{cases}
\]
if the corresponding self-dual quantum code is of \emph{Type~I}, i.e., corresponds to a graph where 
at least one vertex has even degree.
Table~\ref{tab:bounds} also lists the first row of those adjacency matrices that are circulant. The remaining
adjacency matrices are as follows.

\[
\Gamma_{13,1} =
\left(\begin{array}{ccccccccccccc}
0 & 1 & 0 & 0 & 0 & 0 & 0 & 0 & 0 & 1 & 1 & 1 & 1 \\
1 & 0 & 1 & 0 & 0 & 0 & 0 & 1 & 0 & 0 & 0 & 1 & 1 \\
0 & 1 & 0 & 1 & 0 & 1 & 0 & 0 & 0 & 0 & 0 & 1 & 1 \\
0 & 0 & 1 & 0 & 0 & 0 & 1 & 0 & 0 & 0 & 1 & 1 & 1 \\
0 & 0 & 0 & 0 & 0 & 1 & 0 & 1 & 0 & 0 & 1 & 1 & 1 \\
0 & 0 & 1 & 0 & 1 & 0 & 0 & 0 & 0 & 1 & 1 & 0 & 1 \\
0 & 0 & 0 & 1 & 0 & 0 & 0 & 0 & 1 & 1 & 1 & 0 & 1 \\
0 & 1 & 0 & 0 & 1 & 0 & 0 & 0 & 1 & 0 & 1 & 1 & 0 \\
0 & 0 & 0 & 0 & 0 & 0 & 1 & 1 & 0 & 1 & 1 & 1 & 0 \\
1 & 0 & 0 & 0 & 0 & 1 & 1 & 0 & 1 & 0 & 0 & 1 & 0 \\
1 & 0 & 0 & 1 & 1 & 1 & 1 & 1 & 1 & 0 & 0 & 0 & 0 \\
1 & 1 & 1 & 1 & 1 & 0 & 0 & 1 & 1 & 1 & 0 & 0 & 0 \\
1 & 1 & 1 & 1 & 1 & 1 & 1 & 0 & 0 & 0 & 0 & 0 & 0
\end{array}\right)
\]
\[
\Gamma_{13,2} =
\left(\begin{array}{ccccccccccccc}
0 & 1 & 1 & 0 & 0 & 0 & 0 & 0 & 1 & 0 & 0 & 0 & 1 \\
1 & 0 & 0 & 0 & 1 & 0 & 1 & 0 & 0 & 0 & 0 & 1 & 1 \\
1 & 0 & 0 & 1 & 0 & 0 & 0 & 1 & 0 & 0 & 1 & 0 & 0 \\
0 & 0 & 1 & 0 & 0 & 0 & 0 & 0 & 0 & 1 & 1 & 1 & 0 \\
0 & 1 & 0 & 0 & 0 & 1 & 0 & 0 & 0 & 0 & 1 & 1 & 1 \\
0 & 0 & 0 & 0 & 1 & 0 & 1 & 1 & 0 & 1 & 0 & 1 & 0 \\
0 & 1 & 0 & 0 & 0 & 1 & 0 & 0 & 1 & 1 & 0 & 1 & 0 \\
0 & 0 & 1 & 0 & 0 & 1 & 0 & 0 & 1 & 1 & 1 & 0 & 1 \\
1 & 0 & 0 & 0 & 0 & 0 & 1 & 1 & 0 & 1 & 1 & 1 & 1 \\
0 & 0 & 0 & 1 & 0 & 1 & 1 & 1 & 1 & 0 & 1 & 0 & 0 \\
0 & 0 & 1 & 1 & 1 & 0 & 0 & 1 & 1 & 1 & 0 & 1 & 1 \\
0 & 1 & 0 & 1 & 1 & 1 & 1 & 0 & 1 & 0 & 1 & 0 & 1 \\
1 & 1 & 0 & 0 & 1 & 0 & 0 & 1 & 1 & 0 & 1 & 1 & 0
\end{array}\right)
\]
\[
\Gamma_{18} =
\left(\begin{array}{cccccccccccccccccc}
0 & 0 & 1 & 0 & 0 & 0 & 0 & 0 & 1 & 1 & 0 & 0 & 0 & 1 & 0 & 1 & 1 & 1 \\
0 & 0 & 0 & 1 & 1 & 0 & 0 & 0 & 0 & 1 & 0 & 1 & 0 & 0 & 0 & 1 & 1 & 1 \\
1 & 0 & 0 & 0 & 0 & 1 & 0 & 0 & 1 & 1 & 1 & 1 & 0 & 1 & 0 & 0 & 0 & 0 \\
0 & 1 & 0 & 0 & 0 & 0 & 0 & 0 & 1 & 0 & 1 & 0 & 0 & 0 & 1 & 1 & 1 & 1 \\
0 & 1 & 0 & 0 & 0 & 1 & 1 & 1 & 0 & 0 & 0 & 1 & 1 & 0 & 1 & 0 & 0 & 0 \\
0 & 0 & 1 & 0 & 1 & 0 & 0 & 1 & 0 & 0 & 0 & 1 & 0 & 1 & 0 & 1 & 0 & 1 \\
0 & 0 & 0 & 0 & 1 & 0 & 0 & 1 & 1 & 0 & 1 & 0 & 1 & 0 & 1 & 0 & 1 & 0 \\
0 & 0 & 0 & 0 & 1 & 1 & 1 & 0 & 0 & 1 & 1 & 0 & 0 & 1 & 1 & 0 & 0 & 0 \\
1 & 0 & 1 & 1 & 0 & 0 & 1 & 0 & 0 & 0 & 1 & 0 & 1 & 0 & 0 & 0 & 0 & 1 \\
1 & 1 & 1 & 0 & 0 & 0 & 0 & 1 & 0 & 0 & 0 & 1 & 1 & 0 & 0 & 0 & 0 & 1 \\
0 & 0 & 1 & 1 & 0 & 0 & 1 & 1 & 1 & 0 & 0 & 0 & 0 & 1 & 0 & 1 & 0 & 0 \\
0 & 1 & 1 & 0 & 1 & 1 & 0 & 0 & 0 & 1 & 0 & 0 & 1 & 0 & 1 & 1 & 1 & 0 \\
0 & 0 & 0 & 0 & 1 & 0 & 1 & 0 & 1 & 1 & 0 & 1 & 0 & 1 & 0 & 1 & 0 & 0 \\
1 & 0 & 1 & 0 & 0 & 1 & 0 & 1 & 0 & 0 & 1 & 0 & 1 & 0 & 1 & 0 & 0 & 0 \\
0 & 0 & 0 & 1 & 1 & 0 & 1 & 1 & 0 & 0 & 0 & 1 & 0 & 1 & 0 & 0 & 1 & 0 \\
1 & 1 & 0 & 1 & 0 & 1 & 0 & 0 & 0 & 0 & 1 & 1 & 1 & 0 & 0 & 0 & 0 & 0 \\
1 & 1 & 0 & 1 & 0 & 0 & 1 & 0 & 0 & 0 & 0 & 1 & 0 & 0 & 1 & 0 & 0 & 1 \\
1 & 1 & 0 & 1 & 0 & 1 & 0 & 0 & 1 & 1 & 0 & 0 & 0 & 0 & 0 & 0 & 1 & 0
\end{array}\right)
\]
\[
\Gamma_{21} =
\left(\begin{array}{ccccccccccccccccccccc}
0 & 0 & 0 & 0 & 0 & 1 & 0 & 0 & 0 & 1 & 1 & 0 & 1 & 0 & 1 & 0 & 1 & 1 & 1 & 0 & 1 \\
0 & 0 & 1 & 1 & 0 & 0 & 1 & 1 & 0 & 1 & 1 & 0 & 0 & 0 & 1 & 1 & 0 & 0 & 0 & 0 & 0 \\
0 & 1 & 0 & 0 & 1 & 1 & 0 & 0 & 0 & 1 & 0 & 1 & 0 & 0 & 0 & 0 & 0 & 1 & 1 & 1 & 0 \\
0 & 1 & 0 & 0 & 1 & 0 & 1 & 1 & 0 & 0 & 0 & 1 & 1 & 0 & 0 & 0 & 1 & 0 & 1 & 0 & 0 \\
0 & 0 & 1 & 1 & 0 & 1 & 0 & 0 & 0 & 1 & 0 & 0 & 0 & 1 & 1 & 0 & 0 & 0 & 1 & 1 & 0 \\
1 & 0 & 1 & 0 & 1 & 0 & 1 & 1 & 1 & 0 & 0 & 0 & 0 & 1 & 0 & 0 & 0 & 1 & 0 & 1 & 0 \\
0 & 1 & 0 & 1 & 0 & 1 & 0 & 1 & 0 & 1 & 1 & 0 & 1 & 0 & 0 & 0 & 0 & 0 & 1 & 0 & 1 \\
0 & 1 & 0 & 1 & 0 & 1 & 1 & 0 & 0 & 0 & 0 & 0 & 0 & 1 & 0 & 0 & 0 & 1 & 0 & 1 & 1 \\
0 & 0 & 0 & 0 & 0 & 1 & 0 & 0 & 0 & 1 & 1 & 1 & 1 & 1 & 0 & 1 & 0 & 0 & 1 & 0 & 1 \\
1 & 1 & 1 & 0 & 1 & 0 & 1 & 0 & 1 & 0 & 0 & 0 & 0 & 0 & 1 & 1 & 0 & 0 & 0 & 1 & 0 \\
1 & 1 & 0 & 0 & 0 & 0 & 1 & 0 & 1 & 0 & 0 & 0 & 1 & 1 & 0 & 0 & 1 & 0 & 0 & 0 & 1 \\
0 & 0 & 1 & 1 & 0 & 0 & 0 & 0 & 1 & 0 & 0 & 0 & 0 & 1 & 0 & 1 & 1 & 0 & 1 & 0 & 1 \\
1 & 0 & 0 & 1 & 0 & 0 & 1 & 0 & 1 & 0 & 1 & 0 & 0 & 0 & 0 & 1 & 0 & 1 & 0 & 0 & 1 \\
0 & 0 & 0 & 0 & 1 & 1 & 0 & 1 & 1 & 0 & 1 & 1 & 0 & 0 & 0 & 1 & 0 & 1 & 0 & 0 & 0 \\
1 & 1 & 0 & 0 & 1 & 0 & 0 & 0 & 0 & 1 & 0 & 0 & 0 & 0 & 0 & 1 & 1 & 1 & 0 & 0 & 1 \\
0 & 1 & 0 & 0 & 0 & 0 & 0 & 0 & 1 & 1 & 0 & 1 & 1 & 1 & 1 & 0 & 0 & 0 & 0 & 1 & 0 \\
1 & 0 & 0 & 1 & 0 & 0 & 0 & 0 & 0 & 0 & 1 & 1 & 0 & 0 & 1 & 0 & 0 & 1 & 1 & 1 & 0 \\
1 & 0 & 1 & 0 & 0 & 1 & 0 & 1 & 0 & 0 & 0 & 0 & 1 & 1 & 1 & 0 & 1 & 0 & 0 & 0 & 0 \\
1 & 0 & 1 & 1 & 1 & 0 & 1 & 0 & 1 & 0 & 0 & 1 & 0 & 0 & 0 & 0 & 1 & 0 & 0 & 1 & 0 \\
0 & 0 & 1 & 0 & 1 & 1 & 0 & 1 & 0 & 1 & 0 & 0 & 0 & 0 & 0 & 1 & 1 & 0 & 1 & 0 & 0 \\
1 & 0 & 0 & 0 & 0 & 0 & 1 & 1 & 1 & 0 & 1 & 1 & 1 & 0 & 1 & 0 & 0 & 0 & 0 & 0 & 0
\end{array}\right)
\]

For $n=13$ and $n=14$ we were able to compute the interlace polynomial $Q$ of all graphs with optimal $\delta$,
since the corresponding codes have been classified~\cite{selfdual,zlatko}. For other $n$, codes with the same
$\delta$ but with lower $\deg(Q)$ or $Q(G,4)$ may exist.
The best self-dual quantum codes correspond to LC orbits where $\delta$ is maximized, and our results
for graphs on up to $12$ vertices suggested that these LC orbits also minimize $\deg(Q)$ and $Q(G,4)$.
However, in Table~\ref{tab:bounds} we find several examples where the graph we have found
with lowest $\deg(Q)$ does not have maximum $\delta$. 
We have not found a single example where the lowest $Q(G,4)$ is found in a graph
with suboptimal $\delta$, which indicates that $Q(G,4)$ may be a better indicator of the minimum distance of a code
that $\deg(Q)$, and leads to the following conjecture.

\begin{table}
\centering
\caption{Best Found Values of $\delta$, $\deg(Q)$, and $\frac{Q(G,4)}{2^n}$}
\label{tab:bounds}
\begin{tabular}{cccrl}
\toprule
$n$ & $\delta$ & $\deg(Q)$ & $\frac{Q(G,4)}{2^n}$ & Adjacency matrix \\ 
\midrule
13  & $\boldsymbol{4}$         & $\boldsymbol{4}$        & 361                   & $\Gamma_{13,1}$ \\
 13  & $\boldsymbol{4}$         & 5         & $\boldsymbol{360}$                 & $\Gamma_{13,2}$ \\
14  & $\boldsymbol{5}$         & $\boldsymbol{4}$        & $\boldsymbol{549}$    & (00001011101000) \\
15  & $\boldsymbol{5}$         & 6        & $\boldsymbol{830}$                   & (001110011001110) \\ 
 15  & 4         & $\boldsymbol{5}$        & 833                                 & (001111011011110) \\
16  & $\boldsymbol{5}$         & $\boldsymbol{5}$        & $\boldsymbol{1264}$   & (0010101101101010) \\
17  & $\boldsymbol{6}$         & 6        & $\boldsymbol{1872}$                  & (00100011111100010) \\
 17  & 5         & $\boldsymbol{5}$        & 1906                                & (00000111001110000) \\
18  & $\boldsymbol{7}$         & 6        & $\boldsymbol{2808}$                  & $\Gamma_{18}$ \\
 18  & 5         & $\boldsymbol{5}$        & 2835                                & (001001111111110010) \\
19  & $\boldsymbol{6}$         & $\boldsymbol{6}$        & $\boldsymbol{4296}$   & (0000101001100101000) \\
20  & $\boldsymbol{7}$         & $\boldsymbol{6}$        & $\boldsymbol{6444}$   & (00000100111110010000) \\
21  & $\boldsymbol{7}$         & 9        & $\boldsymbol{9672}$                  & $\Gamma_{21}$ \\
 21  & 6         & $\boldsymbol{6}$        & 9756                                & (000001100100100110000) \\
22  & $\boldsymbol{7}$         & $\boldsymbol{6}$        & $\boldsymbol{14,688}$ & (0000001001111100100000) \\
23  & $\boldsymbol{7}$         & 7        & $\boldsymbol{22,013}$                & (00000011101111011100000) \\
 23  & 6         & $\boldsymbol{5}$        & $22,036$                              & (00000111110110111110000) \\
24  & $\boldsymbol{7}$         & $\boldsymbol{6}$        & $\boldsymbol{33,156}$ & (001001110100100101110010) \\
25  & $\boldsymbol{7}$         & 7        & $\boldsymbol{49,812}$                & (0001100001111111100001100) \\
25  & $\boldsymbol{7}$         & $\boldsymbol{6}$        & $49,862$              & (0000011111001100111110000) \\
\bottomrule
\end{tabular}
\end{table}

\begin{con}
Let $G$ be a graph on $n$ vertices, and let $\delta$ be the 
minimum vertex degree over all graphs in the LC orbit of $G$.
If there exists no other graph $G'$ on $n$ vertices 
such that $Q(G',4) < Q(G,4)$, then there exists no other graph on $n$ vertices where the 
minimum vertex degree over all graphs in the LC orbit is greater than $\delta$.
\end{con}

Note that once we have found a graph $G$ on $n$ vertices with a certain $\deg(Q(G))$, 
we can obtain a graph $G'$ on $n-1$ vertices with $\deg(Q(G')) = \deg(Q(G))$ or
$\deg(Q(G')) = \deg(Q(G))-1$ by simply 
deleting any vertex of $G$. This process is equivalent to \emph{shortening} a quantum code~\cite{gaborit},
and it is known that if the minimum vertex degree in the LC orbit of $G$ is $\delta$, then
the minimum vertex degree in the LC orbit of $G'$ is $\delta$ or $\delta-1$.

The following theorem gives an upper bound on $Q(G,4)$ for a graph $G$ with a given value of
$\delta$. Note that the proof relies on certain properties of the 
\emph{aperiodic propagation criteria}~\cite{apc} for Boolean functions, which will not be defined here.

\begin{thm}
\[
Q(G,4) \le \frac{\gamma(\delta+1) + 6^n}{2^n},
\]
where
\[
\gamma(d) = \sum_{t=0}^n {\binom{n}{t}}2^t
\left ( \sum_{k = \max(1,d+t-n)}^t {\binom{t}{k}}2^{n-k} \right ).
\]
\end{thm}
\begin{proof}
The graph $G$ corresponds to a Boolean function with APC distance $d=\delta+1$, 
which means that all fixed-aperiodic autocorrelation coefficients~\cite{apc} up to and including 
weight $d-1$ are set to zero. As the Clifford merit factor (CMF) can be computed with the out-of-phase
sum-of-squares of these autocorrelation coefficients in the denominator, 
then we immediately have a lower bound on CMF dependent on $d$.
For a Boolean function $f$ of $n$ variables with APC distance $d$,
it can thus be shown that the sum-of-squares is upper-bounded by
$\gamma(d)$. The CMF is then lower-bounded by
\[
\text{CMF}(f) \ge \frac{6^n}{\gamma(d)}.
\]
When $f$ is a quadratic Boolean function representing a graph $G$, the upper bound on $Q(G,4)$ 
follows.
\end{proof}

A class of self-dual quantum codes known to have high minimum distance are the \emph{quadratic residue codes}.
The graphs corresponding to these codes are \emph{Paley graphs}. To construct a Paley graph on $n$ vertices, where
$n$ must be a prime power and $n \equiv 1 \pmod{4}$, let the elements of the finite field $\mathbb{F}_n$ be 
the set of vertices, and let two vertices, $i$ and $j$, be joined by an edge if and only if their 
difference is a quadratic residue in $\mathbb{F}_n \backslash \{0\}$,
i.e., there exists an $x \in \mathbb{F}_n \backslash \{0\}$ such that $x^2 \equiv i-j$.
This construction will result in a circulant adjacency matrix, where the first
row is called a \emph{Legendre sequence}.
Paley graphs are known to have low independence numbers, and, since they correspond to strong quantum codes, 
the degrees of their interlace polynomials $Q$ are
also low, i.e., the size of the largest independent set in the LC orbit of a Paley graph is small, 
compared to other graphs on the same number of vertices.
This suggests that Paley graphs, due to their high degree of symmetry, have the 
property that their independence numbers remain largely invariant with respect to LC.
Another code construction is the \emph{bordered quadratic residue code}, equivalent to
extending a Paley graph by adding one vertex and connecting it to all existing vertices.
For example, optimal quantum codes of length 5, 6, 29, and 30 can be constructed using Paley graphs 
or extended Paley graphs.

We have previously discovered~\cite{setapaper} that many strong
self-dual quantum codes can be represented as highly structured \emph{nested clique
graphs}. Some of these graphs are shown in Fig.~\ref{graphs}. For instance, Fig.~\ref{3cl4cl} shows
a graph consisting of three 4-cliques. The remaining
edges form a \emph{Hamiltonian cycle}, i.e., a cycle that visits every vertex of the graph exactly once.
Fig.~\ref{5cl4cl} shows five 4-cliques interconnected by one Hamiltonian cycle and two cycles of length 10.
Ignoring edges in the cliques, there are no cycles of length shorter than 5 in the graph.
The graph in Fig.~\ref{2cl3cl} can be viewed as two interconnected 3-cliques.
Note that the graphs in Fig.~\ref{graphs}
have values of $\delta$, $\deg(Q)$, and $Q(G,4)$ that match 
the optimal or best known values in Tables~\ref{tab:degQ}, \ref{tab:Q4}, and~\ref{tab:bounds}.
Also note that they are all \emph{regular} graphs, with all vertices having degree $\delta$,
which means that the number of edges is minimal for the given $\delta$.

\begin{figure}
 \centering
 \subfloat[{$n=6$, $\delta=3$, $\deg(Q)=2$,\newline$\frac{Q(G,4)}{2^n}=18$}]
 {\hspace{3pt}\includegraphics[width=.41\linewidth]{2cl3cl}\hspace{3pt}\label{2cl3cl}}
 \subfloat[{$n=12$, $\delta=5$, $\deg(Q)=4$,\newline$\frac{Q(G,4)}{2^n}=234$}]
 {\hspace{3pt}\includegraphics[width=.41\linewidth]{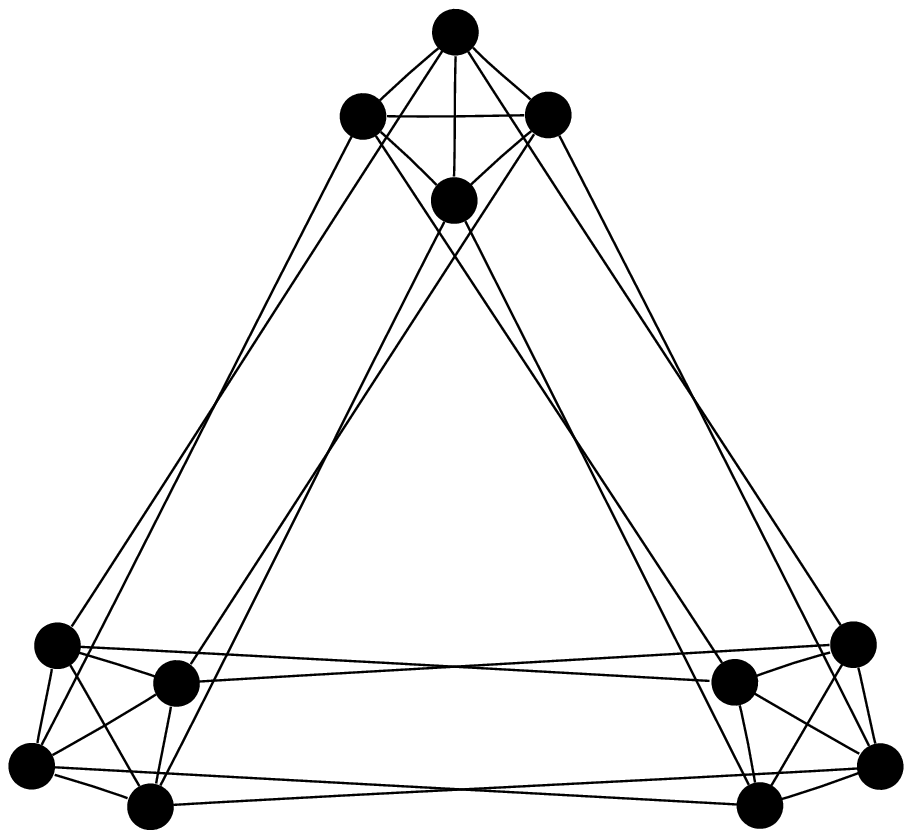}\hspace{3pt}\label{3cl4cl}}
 \subfloat[{$n=20$, $\delta=7$, $\deg(Q)=6$,\newline$\frac{Q(G,4)}{2^n}=6444$}]
 {\hspace{3pt}\includegraphics[width=.41\linewidth]{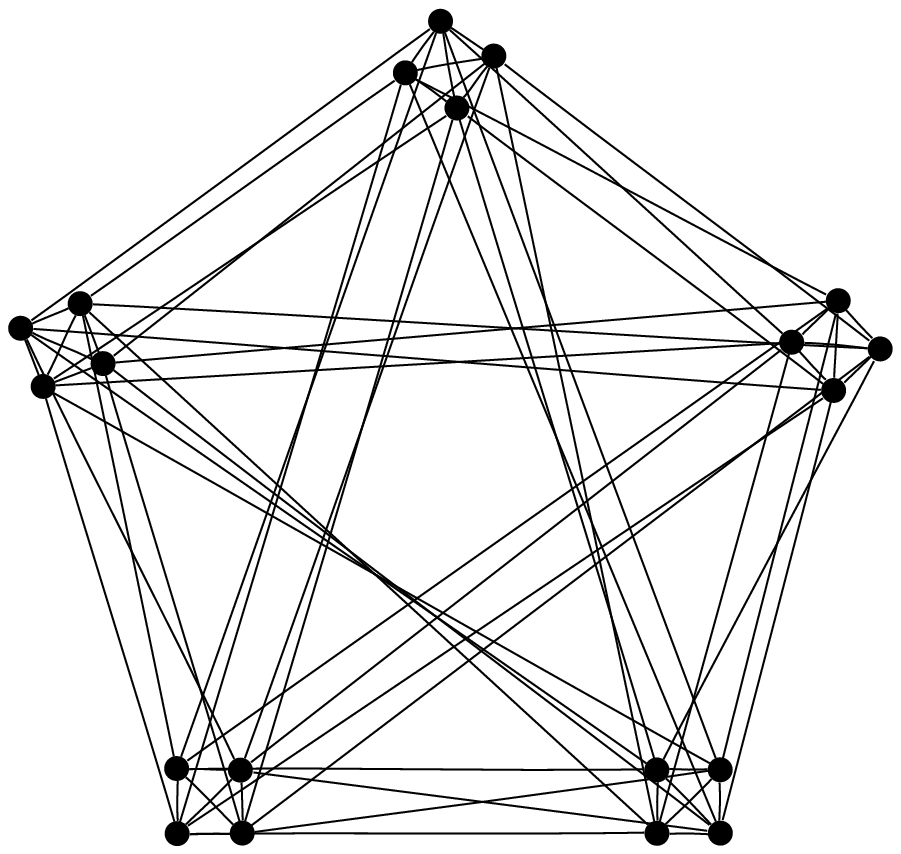}\hspace{3pt}\label{5cl4cl}}
 \caption{Examples of Nested Clique Graphs}\label{graphs}
\end{figure}

It is interesting to observe that the problem of finding good quantum codes, or highly entangled
quantum states, can be reformulated as the problem of finding LC orbits of graphs with certain
properties, and that these properties are related to the interlace polynomials of the graphs.
Even though certain construction techniques are known, as shown above, many open problems
remain, such as providing better bounds on $\delta$, $\deg(Q)$, and $Q(G,4)$, and 
finding new methods for constructing graphs with optimal or good values for these parameters.
It would also be interesting to study possible connections between the observation that the best self-dual quantum
codes have a minimal number of Eulerian subgraphs, and the fact that that many optimal self-dual
quantum codes are of Type~II, i.e., correspond to anti-Eulerian graphs.
Note that all the graphs in Fig.~\ref{graphs} are anti-Eulerian.
The graphs in Fig.~\ref{graphs} also give other clues as to the types of graphs that may optimise
$\deg(Q)$ and $Q(G,4)$.
If a graph contains a $k$-clique, performing LC on any vertex in the clique will produce a
graph with an independent set of size at least $k-1$. Thus the interlace polynomial $Q$ of 
a complete graph will have the highest possible degree of any connected graph. 
This explains why our graphs contain several relatively small cliques.
That the graphs contain a few long cycles reduces the number of cycles in the graph, which makes
sense when we consider that a cycle is an Eulerian subgraph.

It is also possible to say something about which properties should not be present in a graph with 
optimal $\delta$, $\deg(Q)$, or $Q(G,4)$. A bipartite graph on $n$ vertices will have an independence
number of at least $\left\lceil\frac{n}{2}\right\rceil$. Thus the interlace polynomial $Q$ associated 
with an LC orbit that contains a bipartite graph will have degree at least $\left\lceil\frac{n}{2}\right\rceil$.
Note that bipartiteness is preserved by ELC, but not by LC.
In Table~\ref{tab:bipbydistance}, we give the number of LC orbits containing connected bipartite graphs 
on $n$ vertices with a given value of $\delta$. Compare this to Table~\ref{tab:allbydistance}, which includes
LC orbits of all connected graphs.
It also turns out that circle graphs are bad. This is not surprising, given that the 
circle graph obstructions shown in Fig.~\ref{fig:obstructions} all have optimal values of $\delta$. The
obstruction on 6 vertices also has optimal value of $Q(G,4)$, and the two other obstructions have $Q(G,4)$ only
one greater than optimal.
In Table~\ref{tab:circlebydistance}, we give the number of LC orbits of connected circle graphs 
on $n$ vertices with a given value of $\delta$. 

\begin{table}
\centering
\caption{Number of LC Orbits Containing Connected Bipartite Graphs by $\delta$ and $n$}
\label{tab:bipbydistance}
\begin{tabular}{rrrrrrrrrrrr}
\toprule
$\delta \backslash n$ & 
2 & 3 & 4 & 5 & 6 & 7 & 8 & 9 & 10 & 11 & 12 \\
\midrule
1     & 1 & 1 & 2 &  3 &  7 & 14 &  40 & 106 &   352 &   1218 &      5140 \\
2     &   &   &   &    &  1 &  1 &   2 &   4 &    16 &     41 &       215 \\
3     &   &   &   &    &    &    &   1 &     &     2 &      1 &        11 \\
\midrule
All   & 1 & 1 & 2 &  3 &  8 & 15 &  43 & 110 &   370 &   1260 &      5366 \\
\bottomrule
\end{tabular}
\end{table}

\begin{table}
\centering
\caption{Number of LC Orbits of Connected Circle Graphs by $\delta$ and $n$}
\label{tab:circlebydistance}
\begin{tabular}{rrrrrrrrrrrr}
\toprule
$\delta \backslash n$ & 
2 & 3 & 4 & 5 & 6 & 7 & 8 & 9 & 10 & 11 & 12 \\
\midrule
1     & 1 & 1 & 2 &  3 &  9 & 21 &  75 & 277 &  1346 &   7712 &    54,067 \\
2     &   &   &   &  1 &  1 &  2 &   5 &  16 &    55 &    254 &      1474 \\
3     &   &   &   &    &    &    &   1 &     &     2 &      2 &        12 \\
\midrule
All   & 1 & 1 & 2 &  4 & 10 & 23 &  81 & 293 &  1403 &   7968 &    55,553 \\
\bottomrule
\end{tabular}
\end{table}

\begin{table}
\centering
\caption{Number of LC Orbits of Connected Graphs by $\delta$ and $n$}
\label{tab:allbydistance}
\begin{tabular}{rrrrrrrrrrrr}
\toprule
$\delta \backslash n$ & 
2 & 3 & 4 & 5 & 6 & 7 & 8 & 9 & 10 & 11 & 12 \\
\midrule
1     & 1 & 1 & 2 &  3 &  9 & 22 &  85 & 363 &  2436 & 26,750 &   611,036 \\
2     &   &   &   &  1 &  1 &  4 &  11 &  69 &   576 & 11,200 &   467,513 \\
3     &   &   &   &    &  1 &    &   5 &   8 &   120 &   2506 &   195,455 \\
4     &   &   &   &    &    &    &     &     &       &      1 &        63 \\
5     &   &   &   &    &    &    &     &     &       &        &         1 \\
\midrule
All   & 1 & 1 & 2 &  4 & 11 & 26 & 101 & 440 &  3132 & 40,457 & 1,274,068 \\
\bottomrule
\end{tabular}
\end{table}

\paragraph*{Acknowledgements}
This research was supported by the Research Council of Norway.

\end{document}